\documentclass[10pt,draft]{article}
\usepackage{amsmath,amscd}
\usepackage{amssymb,latexsym,amsthm}
\usepackage{color}
\usepackage[spanish,english]{babel}
\usepackage{amssymb}
\usepackage{color}
\usepackage{amsmath,amsthm,amscd}
\usepackage[latin1]{inputenc}
\usepackage{lscape}
\usepackage{fancyhdr}
\usepackage{amsfonts}
\usepackage{pb-diagram}

\numberwithin{equation}{section}
\newtheorem{theorem}{Theorem}[section]

\newtheorem{definition}[theorem]{Definition}
\newtheorem{proposition}[theorem]{Proposition}

\newtheorem{corollary}[theorem]{Corollary}

\theoremstyle{definition}

\newtheorem{example}[theorem]{Example}

\newtheorem{remark}[theorem]{Remark}

\newcommand{\cU}{\mbox{${\cal U}$}}

\newcommand{\cW}{\mbox{${\cal W}$}}

\title{\textbf{Computation of point modules\\ of finitely semi-graded rings}}
\author{Oswaldo Lezama\\
\texttt{jolezamas@unal.edu.co}
\\ Seminario de Álgebra Constructiva - SAC$^2$\\ Departamento de Matemáticas\\ Universidad Nacional de
Colombia, Sede Bogotá}
\date{}
\begin{document}
\maketitle
\begin{abstract}
\noindent In this paper we compute the set of point modules of
finitely semi-graded rings. In particular, from the
parametrization of the point modules for the quantum affine
$n$-space, the set of point modules for some important examples of
non $\mathbb{N}$-graded quantum algebras is computed.

\bigskip


\bigskip

\end{abstract}

\section{Introduction}

In algebraic geometry a key role play the point modules and its
parametrization. The point modules of algebras, and the spaces
parameterizing them, were first studied by Artin, Tate and Van den
Bergh (\cite{Artin1}) in order to complete the classification of
Artin-Schelter regular algebras of dimension $3$. Many commutative
and non-commutative graded algebras have nice parameter spaces of
point modules. For example, let $K$ be a field and
$A=K[x_0,\dots,x_{n}]$ be the commutative $K$-algebra of usual
polynomials, then there exists a bijective correspondence between
the projective space $\mathbb{P}^n$ and the set of isomorphism
classes of point modules of $A$; this correspondence generalizes
to any finitely graded commutative algebra generated in degree
$1$. Similarly, the point modules of the quantum plane and the
Jordan plane are parametrized by $ \mathbb{P}^{1}$. The study of
the point modules and its parametrization has been concentrated in
finitely graded algebras (see \cite{Artin1}, \cite{Artin4},
\cite{Goetz}, \cite{Kanazawa}, \cite{Rogalski}, \cite{Smith}), in
the present paper we are interested in the computation of point
modules for algebras that are not necessarily $\mathbb{N}$-graded,
examples of such algebras are the quantum algebra
$\mathcal{U}'(so(3,K))$, the dispin algebra $\cU(osp(1,2))$, the
Woronowicz algebra $\cW_{\nu}(\mathfrak{sl}(2,K))$, among many
others. In \cite{lezamalatorre} were introduced the semi-graded
rings and for them was proved a classical theorem of
non-commutative algebraic geometry, the Serre-Artin-Zhang-Verevkin
theorem about non-commutative projective schemes (see also
\cite{Lezama}). Thus, the semi-graded rings are nice objects for
the investigation of geometric properties of non-commutative
algebras that are not $\mathbb{N}$-graded, and include as a
particular case the class of finitely graded algebras. In
\cite{Lezama-Gomez} was initiated the study the point modules of
finitely semi-graded rings, in the present paper we complete this
investigation and  we will compute the point modules of finitely
semi-graded rings generated in degree one. In particular, from the
parametrization of the point modules for the quantum affine
$n$-space, the set of point modules for some important examples of
non $\mathbb{N}$-graded quantum algebras is computed

The paper is organized as follows: In the first section we recall
some basic facts and examples on point modules for classical
$\mathbb{N}$-graded algebras, we include a complete proof (usually
not available in the literature) of the parametrization of the
point modules for the quantum affine $n$-space (Example
\ref{remark17.6.11}); we review the definition and elementary
properties of semi-graded rings, in particular, the subclass of
finitely semi-graded rings and modules. As was pointed out above,
the finitely semi-graded rings generalize the finitely graded
algebras, so in order present sufficient examples of finitely
semi-graded rings not being necessarily finitely graded algebras,
we recall in the last part of the first section the notion of skew
$PBW$ extension defined firstly in \cite{LezamaGallego}. Skew
$PBW$ extensions represent a way of describing many important
non-commutative algebras not necessarily $\mathbb{N}$-graded,
remarkable examples are the operator algebras, algebras of
diffusion type, quadratic algebras in $3$ variables, the quantum
algebra $\mathcal{U}'(so(3,K))$, the dispin algebra
$\cU(osp(1,2))$, the Woronowicz algebra
$\cW_{\nu}(\mathfrak{sl}(2,K))$, among many others. The second and
the third sections contain the novelty of the paper. The main
results are Theorem \ref{theorem17.6.8}, which completes the
parametrization started in \cite{Lezama-Gomez}, and Theorem
\ref{corollary17.6.9} where we apply the results of the previous
sections to compute the point modules of some skew $PBW$
extensions, and hence, the point modules of some examples of non
$\mathbb{N}$-graded quantum algebras. Concrete examples are
presented in Examples \ref{example3.2} and \ref{example3.3}.

\subsection{Point modules for finitely graded algebras}

In this first subsection we recall some key facts and examples on
point modules for classical $\mathbb{N}$-graded algebras that we
will use later in the paper (see \cite{Artin1}, \cite{Artin4} and
\cite{Rogalski}). Let $K$ be a field, a $K$-algebra $A$ is
\textit{finitely graded} if the following conditions hold: (i) $A$
is $\mathbb{N}$-graded, $A=\bigoplus_{n\geq 0}A_n$. (ii) $A$ is
connected, i.e., $A_0=K$. (iii) $A$ is finitely generated as
$K$-algebra. Let $A$ be a finitely graded $K$-algebra that is
generated in degree $1$. A \textit{point module} for $A$ is a
graded left module $M$ such that $M$ is cyclic, generated in
degree $0$, i.e., there exists an element $m_0\in M_0$ such that
$M=Am_0$, and $\dim_K(M_n)=1$ for all $n\geq 0$ (since $A$ is
$\mathbb{N}$-graded and $m_0\in M_0$, then $M$ is necessarily
$\mathbb{N}$-graded). The collection of isomorphic classes of
point modules for $A$ is denoted by $P(A)$.

The following examples show the description (parametrization) of
point modules for some well-known finitely graded algebras.

\begin{example}[\cite{Rogalski}]\label{17.3.2}
Let $K$ be a field and $A=K[x_0,\dots,x_{n}]$ be the commutative
$K$-algebra of usual polynomials. Then,
\begin{enumerate}
\item[\rm (i)]There exists a bijective correspondence between the
projective space $\mathbb{P}^n$ and the set of isomorphism classes
of point modules of $A$, $\mathbb{P}^n\longleftrightarrow P(A)$.
\item[\rm (ii)]The correspondence in $(i)$ generalizes to any finitely graded
commutative algebra which is generated in degree $1$.
\end{enumerate}
\end{example}

\begin{example}[\cite{Rogalski}]\label{exampleA.3.4}
(i) The point modules of the quantum plane are parametrized by
$\mathbb{P}^1$, i.e., there exists a bijective correspondence
between the projective space $\mathbb{P}^1$ and the collection of
isomorphism classes of point modules of the quantum plane $A:=K\{
x,y\}/\langle yx-qxy\rangle$:
\begin{center}
$\mathbb{P}^1\longleftrightarrow P(A)$.
\end{center}
(ii) For the Jordan plane, $A:=K\{ x,y\}/\langle
yx-xy-x^2\rangle$, can be proved a similar equivalence:
\begin{center}
$\mathbb{P}^1\longleftrightarrow P(A)$.
\end{center}
\end{example}

\begin{example}[\cite{Rogalski}]\label{example16.1.34}
The isomorphism classes of point modules for the free algebra
$A:=K\{x_0,\dots,x_n\}$ are in bijective correspondence with
$\mathbb{N}$-indexed sequences of points in $\mathbb{P}^n$,
$\{(\lambda_{0,i}:\cdots :\lambda_{n,i})\in \mathbb{P}^n|i\geq
0\}$, in other words, with the points of the infinite product
$\mathbb{P}^n\times \mathbb{P}^n\times \cdots=\prod_{i=0}^\infty
\mathbb{P}^n$(note that an element of $\prod_{i=0}^\infty
\mathbb{P}^n$ is a sequence of the form $((\lambda_{0,i}:\cdots
:\lambda_{n,i}))_{i\geq 0})$. Thus, there exists a bijective
correspondence between $\prod_{i=0}^\infty \mathbb{P}^n$ and
$P(K\{x_0,\dots,x_n\})$:
\begin{center}
$\prod_{i=0}^\infty \mathbb{P}^n\longleftrightarrow P(A)$.
\end{center}
\end{example}

Using the previous example it is possible to compute the
collection of point modules for any finitely presented algebra.
For this we need a definition first, see \cite{Rogalski}. Suppose
that $f\in K\{x_0,\dots,x_n\}$ is a homogeneous element of degree
$m$ $($assuming that $\deg(x_v):=1$, for all $v$$)$. Consider the
polynomial ring $B=K[y_{v,u}]$ in a set of $(n+1)m$ commuting
variables $\{y_{v,u}|0\leq v\leq n,1\leq u\leq m\}$. The
\textit{multilinearization} of $f$ is the element of $B$ given by
replacing each word $w = x_{v_m}\cdots x_{v_2}x_{v_1}$ occurring
in $f$ by $y_{v_m,m}\cdots
y_{v_2,2}y_{v_1,1}(=y_{v_1,1}y_{v_2,2}\cdots y_{v_m,m})$.

\begin{proposition}[\cite{Rogalski}, Proposition 3.5]\label{17.3.6}
Let $A=K\{x_0,\dots,x_n\}/\langle f_1,\dots,f_r\rangle$ be a
finitely presented $K$-algebra, where the $f_l$ are homogeneous of
degree $d_l\geq 2$, $1\leq l\leq r$. For each $f_l$, let $g_l$ be
the multilinearization of $f_l$. Then,
\begin{enumerate}
\item[\rm (i)]The isomorphism classes of point modules for $A$ are
in bijection with the closed subset $X$ of $\prod_{i=0}^\infty
\mathbb{P}^n$ defined by
\begin{center}
$X:=\{(p_0,p_1,\dots )|g_l(p_i,p_{i+1},\dots,p_{i+d_l-1})=0\
\text{for all}\ 1\leq l\leq r, i\geq 0\}$.
\end{center}
\end{enumerate}
\item[\rm (ii)]Consider for each $m\geq 1$ the closed subset
\begin{center}
{\footnotesize $X_m:=\{(p_0,p_1,\dots
,p_{m-1})|g_l(p_i,p_{i+1},\dots,p_{i+d_l-1})=0\ \text{for all}\
1\leq l\leq r, 0\leq i\leq m-d_l\}$}
\end{center}
of $\prod_{i=0}^{m-1}\mathbb{P}^n$. The canonical projection onto
the first m coordinates defines a function $\phi_m:X_{m+1}\to
X_m$. Then $X$ is equal to the inverse limit
$\underleftarrow{\lim}X_m$ of the $X_m$ with respect to the
functions $\phi_m$. In particular, if $m_0$ is such that $\phi_m$
is a bijection for all $m\geq m_0$, then the isomorphism classes
of point modules of $A$ are in bijective correspondence with the
points of $X_{m_0}$.
\end{proposition}

For strongly Noetherian algebras there is $m_0$ such that the
functions $\phi_m:X_{m+1}\to X_m$ are bijective for all $m\geq
m_0$. These algebras were studied by Artin, Small and Zhang in
\cite{Artin2}, and appears naturally in the study of point modules
in non-commutative algebraic geometry (see \cite{Artin1},
\cite{Artin3} and \cite{Rogalski}). Let $K$ be a field and let $A$
be a left Noetherian $K$-algebra, it is said that $A$ is
\textit{left strongly Noetherian} if for any commutative
Noetherian $K$-algebra $C$, $C\otimes_{K}A$
 is left Noetherian.

\begin{corollary}[\cite{Artin4}, Corollary E4.12]\label{corollaryA.3.8}
Let $A$ be a finitely graded strongly noetherian algebra, with
presentation as in Proposition \ref{17.3.6}. Then there is $m_0$
such that $\phi_m:X_{m+1}\to X_m$ is bijective for all $m\geq
m_0$, and the point modules of $A$ are in bijective correspondence
with the points of $X_{m_0}$.
\end{corollary}

\begin{example}\label{example17.6.11}
It is known (see  Example 3.8 in \cite{Rogalski} and also
\cite{Kanazawa}) that for the mul\-ti-\-pa\-ra\-meter quantum
affine $3$-space, $A:=K_{\textbf{q}}[x_1,x_2,x_3]$, defined by
\begin{center}
$x_2x_1=q_{12}x_1x_2$, $x_3x_1=q_{13}x_1x_3$,
$x_3x_2=q_{23}x_2x_3$,
\end{center}
$m_0=2$ and
\begin{center}
$X_2=\{(p_0,\sigma(p_0))\mid p_0\in E\}$,
\end{center}
where
\begin{equation}\label{equation17.6.10}
E=
\begin{cases}
\mathbb{P}^2 \ \text{if}\ q_{12}q_{23}=q_{13}, & \\
\{(x_0:y_0:z_0)\in \mathbb{P}^2\mid x_0y_0z_0=0\}\ \text{if}\
q_{12}q_{23}\neq q_{13},
\end{cases}
\end{equation}
with $\sigma:E\to E$ bijective. Thus, for $A$ we have the
bijective correspondence $P(A)\longleftrightarrow
X_2\longleftrightarrow E$. We will prove this including all
details omitted in the above cited references. Following
Proposition \ref{17.3.6}, in this case we have $3$ variables,
$r=3$ and every generator of $A$ has degree $2$. We have to show
that the description of $X_2$ is as above and for all $m\geq 2$
the function $\phi_m:X_{m+1}\to X_m$ is bijective.

Recall that
\begin{center}
$X_2=\{(p_0,p_1)\in \mathbb{P}^2\times \mathbb{P}^2\mid
g_l(p_0,p_1)=0\ \text{for}\ 1\leq l\leq 3\}$,
\end{center}
where $p_0:=(x_0:y_0:z_0)$, $p_1:=(x_1:y_1:z_1)$ and $g_1,g_2,g_3$
are the multilinearization of the defining relations of $A$, i.e.,
\begin{center}
$g_1(p_0,p_1)=y_0x_1-q_{12}x_0y_1$,
$g_2(p_0,p_1)=z_0x_1-q_{13}x_0z_1$,
$g_3(p_0,p_1)=z_0y_1-q_{23}y_0z_1$.
\end{center}
Thus, the relations that define $A$ can be written as
\begin{center}
$y_0x_1-q_{12}x_0y_1=0$, $z_0x_1-q_{13}x_0z_1=0$ and
$z_0y_1-q_{23}y_0z_1=0$,
\end{center}
and in a matrix form we have
\begin{equation}\label{equation17.6.1}
\begin{bmatrix} y_0 & -q_{12}x_0 & 0\\
z_0 & 0 & -q_{13}x_0\\
0 & z_0 & -q_{23}y_0
\end{bmatrix}
\begin{bmatrix}
x_1\\
y_1\\
z_1
\end{bmatrix}
=0.
\end{equation}
In order to compute $X_2$, let $E$ be the projection of $X_2$ onto
the first copy of $\mathbb{P}^2$. Let $p_0\in \mathbb{P}^2$, then
$p_0\in E$ if and only if (\ref{equation17.6.1}) has a non trivial
solution $p_1$ if and only if the determinant of the matrix $F$ of
this system is equal $0$, i.e.,
$x_0y_0z_0(q_{12}q_{23}-q_{13})=0$. In addition, observe that
${\rm rank}(F)=2$. In fact, it is clear that ${\rm rank}(F)\neq 0$
(contrary, $x_0=y_0=z_0=0$, false). Suppose that ${\rm
rank}(F)=1$, and assume for example that the $K$-basis of $F$ is
the first column, so $0=z_0=y_0=x_0$, a contradiction; in a
similar way we get a contradiction assuming that the $K$-basis of
$F$ is the second or the third column. Thus, $\dim_K(\ker(F))=1$
and given $p_0=(x_0:y_0:z_0)\in E$ there exists a unique
$p_1=(x_1:y_1:z_1)\in \mathbb{P}^2$ that satisfies
(\ref{equation17.6.1}), and hence, we define a function
\begin{equation}\label{equation17.6.2}
\sigma:E \to \mathbb{P}^2, \  \sigma(p_0):=p_1.
\end{equation}
Therefore, $X_2=\{(p_0,\sigma(p_0))\mid p_0\in E\}$.

%
%

$\sigma$ is injective: For this, note that the defining relations
of $A$ can be written also in the following matrix way:
\begin{center}
$\begin{bmatrix} -q_{12}y_1 & x_1 & 0\\
-q_{13}z_1 & 0 & x_1\\
0 & -q_{23}z_1 & y_1
\end{bmatrix}
\begin{bmatrix}
x_0\\
y_0\\
z_0
\end{bmatrix}
=0$;
\end{center}
let $G$ be the matrix of this system; as above, ${\rm rank}(G)=2$
and $\dim_K(\ker(G))=1$, thus, if $\sigma(p_0)=p_1=\sigma(p_0')$,
then $p_0=p_0'$.

From the determinant of $F$ arise two cases.

\textit{Case 1}: $q_{12}q_{23}=q_{13}$. In this case every point
$p_0=(x_0:y_0:z_0)$ of $\mathbb{P}^2$ is in $E$, i.e.,
$E=\mathbb{P}^2$. Let $\sigma_1:\mathbb{P}^2\to \mathbb{P}^2$ be
the function in (\ref{equation17.6.2}), we have to show that
$\sigma_1$ is surjective. Since in this case $p_0\in E$ if and
only if $p_1\in E$,  we define $\theta_1:\mathbb{P}^2\to
\mathbb{P}^2$ by $\theta_1(p_1):=p_0$ with $Gp_0=0$. Observe that
$\sigma_1\theta_1=i_{\mathbb{P}^2}$: Indeed,
$\sigma_1(\theta_1(p_1))=\sigma_1(p_0)=p_1'$ with $Fp_1'=0$, but
$Gp_0=0$ is equivalent to $Fp_1=0$, and since $\dim_K(\ker(F))=1$,
then $p_1'=p_1$.

\textit{Case 2}: $q_{12}q_{23}\neq q_{13}$. In this case
$p_0=(x_0:y_0:z_0)\in E$ if and only if $x_0y_0z_0=0$, i.e.,
$E=\{(x_0:y_0:z_0)\in \mathbb{P}^2\mid x_0y_0z_0=0\}$. We denote
by $\sigma_2$ the function in (\ref{equation17.6.2}). We know that
$\sigma_2$ is injective; observe that $\sigma_2(E)\subseteq E$:
Indeed, let $p_0\in E$, then $\sigma_2(p_0)=\sigma(p_0)=p_1$, but
since $q_{12}q_{23}\neq q_{13}$ and $\det(G)=0$ (contrary,
$x_0=y_0=z_0=0$, false), then $x_1y_1z_1=0$, i.e., $p_1\in E$.
Only rest to show that $\sigma_2:E\to E$ is surjective. As in the
case 1, we define $\theta_2:E\to E$ by $\theta_2(p_1):=p_0$ with
$Gp_0=0$, then $\sigma_2(\theta_2(p_1))=\sigma_2(p_0)=p_1'$, with
$Fp_1'=0$, but $Gp_0=0$ is equivalent to $Fp_1=0$, and since
$\dim_K(\ker(F))=1$, then $p_1'=p_1$.

Thus, in both cases we have proved that
\begin{center}
$X_2=\{(p_0,\sigma(p_0))\mid p_0\in E\}\longleftrightarrow E$.
\end{center}
To complete the example, we will show that $X\longleftrightarrow
X_2$. According to Proposition \ref{17.3.6}, we have to prove that
for every $m\geq 2$, $\phi_m:X_{m+1}\to X_m$ is a bijection. For
$m=2$, $\phi_2(p_0,p_1,p_2)=(p_0,p_1)\in X_2$, so
$p_1=\sigma(p_0)$, note that $p_2=\sigma^2(p_0)$ since by the
definition of $X$, $g_l(p_1,p_2)=0$ for $1\leq l\leq 3$, so by the
proved above $(\sigma(p_0),\sigma^2(p_0))\in X_2$. Thus,
$X_3=\{(p_0,\sigma(p_0), \sigma^2(p_0))\mid p_0\in E\}$ and
$\phi_2$ is bijective. By induction, assume that
\begin{center}
$X_m=\{(p_0,\sigma(p_0),\dots, \sigma^{m-1}(p_0))\mid p_0\in E\}$
and $X_m\longleftrightarrow X_{m-1}$,
\end{center}
so $X_{m+1}=\{(p_0,\sigma(p_0),\dots, \sigma^{m}(p_0))\mid p_0\in
E\}$ since $(\sigma^{m-1}(p_0),\sigma^{m}(p_0))\in X_2$, and hence
$X_{m+1}\longleftrightarrow X_m$.
\end{example}

\begin{example}\label{remark17.6.11}
We will show next that the ideas and results of Example
\ref{example17.6.11} can be extended to the
mul\-ti-\-pa\-ra\-meter quantum affine $n$-space
$A:=K_{\textbf{q}}[x_1,\dots,x_n]$, with $n\geq 3$ (see also
\cite{Kanazawa}). We will see that $m_0=2$ and $X_2$ is as in
(\ref{equation17.6.4}) below.

Recall that in $A$ we have the defining relations
\begin{center}
$x_jx_i=q_{ij}x_ix_j$, $1\leq i<j\leq n$.
\end{center}
We have $n$ variables, $r=\frac{n(n-1)}{2}$ generators, every
generator of $A$ has degree $2$ and
\begin{center}
$X_2=\{(p_0,p_1)\in \mathbb{P}^{n-1}\times \mathbb{P}^{n-1}\mid
g_l(p_0,p_1)=0\ \text{for}\ 1\leq l\leq r\}$,
\end{center}
where $p_0:=(x_1^0:\cdots :x_n^0)$, $p_1:=(x_1^1:\cdots :x_n^1)$
and the $g_l$ are the multilinearization of the defining relations
of $A$. Thus, the relations that define $A$ can be written as
\begin{center}
$x_j^0x_i^1-q_{ij}x_i^0x_j^1=0$, $1\leq i<j\leq n$,
\end{center}
and in a matrix form we have
\begin{equation*}
\begin{bmatrix} x_2^0 & -q_{12}x_1^0 & 0 & 0 & 0 & \cdots & 0 & 0\\
x_3^0 & 0 & -q_{13}x_1^0 & 0 & 0 & \cdots & 0 & 0 \\
x_4^0 & 0 & 0 & -q_{14}x_1^0 & 0 & \cdots & 0 & 0 \\
\vdots & \vdots & \vdots & \vdots & \vdots &  & \vdots & \vdots \\
x_n^0 & 0 & 0 & 0 & 0 & \cdots & 0 & -q_{1n}x_1^0\\
0 & x_3^0 & -q_{23}x_2^0 & 0 & 0 & \cdots & 0 & 0\\
0 & x_4^0 & 0 & q_{24}x_2^0 & 0 & \cdots & 0 & 0\\
0 & x_5^0 & 0 & 0 & q_{25}x_2^0 & \cdots & 0 & 0\\
\vdots & \vdots & \vdots & \vdots & \vdots &  & \vdots & \vdots \\
0 & x_n^0 & 0 & 0 & 0 & \cdots & 0 & -q_{2n}x_2^0\\
\vdots & \vdots & \vdots & \vdots & \vdots &  & \vdots & \vdots \\
0 & 0 & 0 & 0 & 0 & \cdots & x_n^0 & -q_{n-1,n}x_{n-1}^0\\
\end{bmatrix}
\begin{bmatrix}
x_{1}^1\\
x_2^1\\
\vdots\\
\\
\vdots \\
\\
\vdots\\
x_{n-1}^1\\
x_{n}^1
\end{bmatrix}
=0.
\end{equation*}
The size of the matrix $F$ of this system is
$\frac{n(n-1)}{2}\times n$. In order to compute $X_2$, let $E$ be
the projection of $X_2$ onto the first copy of $\mathbb{P}^{n-1}$.
Let $p_0\in \mathbb{P}^{n-1}$, then $p_0\in E$ if and only if the
previous system has a non trivial solution $p_1$. This last
condition is equivalent to ${\rm rank}(F)=n-1$: Indeed, it is
clear that ${\rm rank}(F)\leq n$, but ${\rm rank}(F)\neq n$
(contrary, ${\rm dim}_K(\ker(F))=0$ and hence $x_1^1=\cdots
x_n^1=0$, false); thus, ${\rm rank}(F)\leq n-1$, but there exists
$i\in \{1,\dots,n\}$ such that $x_i^0\neq 0$ and the following
$n-1$ rows of $F$ are linearly independent (actually, this is true
for every $x_i^0\neq 0$):
\begin{center}
$\begin{bmatrix} \cdots x_i^0 & \cdots & -q_{li}x_i^0 & \cdots
\end{bmatrix}$ and $\begin{bmatrix} \cdots & \cdots &  x_k^0 & \cdots & -q_{ik}x_i^0 & \cdots
\end{bmatrix}$
\end{center}
for all $1\leq l<i<k\leq n$, where $x_i^0$ is in the $l$-position,
$-q_{li}x_i^0$ is in the $i$-position, $x_k^0$ is in the
$i$-position and $-q_{ik}x_i^0$ is in the $k$-position. Hence,
${\rm rank}(F)=n-1$.  Conversely, if ${\rm rank}(F)=n-1$, then the
system has non trivial solution.

Thus, $p_0\in E$ if and only if ${\rm rank}(F)=n-1$ if and only if
${\rm dim}_K(\ker(F))=1$, and hence, given $p_0\in E$ there exists
a unique $p_1\in \mathbb{P}^{n-1}$ that satisfies the above matrix
system, so we define a function
\begin{equation*}
\sigma:E \to \mathbb{P}^{n-1}, \  \sigma(p_0):=p_1.
\end{equation*}
Therefore, $X_2=\{(p_0,\sigma(p_0))\mid p_0\in E\}$. Rewriting the
defining relations of $A$ as we did in Example
\ref{example17.6.11}, we conclude that $\sigma$ is injective.
Moreover, $Im(\sigma)=E$: In fact, recall that ${\rm rank}(F)=n-1$
if and only if every minor of $F$ of size $n$ is equal $0$; let
$\mathcal{F}$ be the set of minors of size $n$ of the matrix $F$,
then $E$ is the projective variety
\begin{equation}\label{equation17.6.4}
E=\bigcap_{f\in \mathcal{F}}V(f)=V(I_F),
\end{equation}
where $I_F$ is the ideal of
$K[x_1^0,\dots,x_n^0]=K[x_1,\dots,x_n]$ generated by all $f\in
\mathcal{F}$. Hence, from the rewriting of relations we conclude
also that $p_1=\sigma(p_0)\in E$ and, as in Example
\ref{example17.6.11}, we define a function $\theta:E\to E$ such
that $\sigma\theta=i_E$.


Finally, $X_2=\{(p_0,\sigma(p_0))\mid p_0\in
E\}\longleftrightarrow E$ and the proof of $X\longleftrightarrow
X_2$ is exactly as in the final part of Example
\ref{example17.6.11}.
\end{example}

\begin{remark}
(i) For $n=3$,
$\mathcal{F}=\{x_1^0x_2^0x_3^0(q_{12}q_{23}-q_{13})\}$ and hence
(\ref{equation17.6.4}) extends (\ref{equation17.6.10}). For $n=2$,
$A=K_q[x_1,x_2]$ is the quantum plane and from Example
\ref{exampleA.3.4} we know that
$P(A)\longleftrightarrow\mathbb{P}^1$, but observe that
(\ref{equation17.6.4}) also cover this case since
$\mathcal{F}=\emptyset$ and hence $E=V(0)=\mathbb{P}^1$.

(ii) For the affine algebra $A=K[x_1,\dots,x_n]$ all constants
$q_{ij}$ are trivial, i.e, $q_{ij}=1$ and the generators of $I_F$
in (\ref{equation17.6.4}) are null, so $I_F=0$ and
$E=\mathbb{P}^{n-1}$. This result agrees with Example
\ref{17.3.2}.

(iii) It is important to recall that there exist quotients of
finitely graded algebras $A$ for which $P(A)=\emptyset$. In fact,
some quotient algebras of the  mul\-ti-\-pa\-ra\-meter quantum
affine $n$-space have empty set of point modules, see
\cite{Vancliff1} and \cite{Vancliff2}.
\end{remark}

\subsection{Finitely semi-graded rings}

In this preliminary subsection we recall the definition of
semi-graded rings and modules introduced firstly in
\cite{lezamalatorre}. Let $B$ be a ring. We say that $B$ is
\textit{semi-graded} $(SG)$ if there exists a collection
$\{B_n\}_{n\geq 0}$ of subgroups $B_n$ of the additive group $B^+$
such that the following conditions hold:
\begin{enumerate}
\item[\rm (i)]$B=\bigoplus_{n\geq 0}B_n$.
\item[\rm (ii)]For every $m,n\geq 0$, $B_mB_n\subseteq B_0\oplus \cdots \oplus B_{m+n}$.
\item[\rm (iii)]$1\in B_0$.
\end{enumerate}
The collection $\{B_n\}_{n\geq 0}$ is called a
\textit{semi-graduation} of $B$ and we say that the elements of
$B_n$ are \textit{homogeneous} of degree $n$. Let $B$ and $C$ be
semi-graded rings and let $f: B\to C$ be a ring homomorphism, we
say that $f$ is \textit{homogeneous} if $f(B_n)\subseteq C_{n}$
for every $n\geq 0$. Let $B$ be a $SG$ ring and let $M$ be a
$B$-module. We say that $M$ is a $\mathbb{Z}$-semi-graded, or
simply semi-graded, if there exists a collection $\{M_n\}_{n\in
\mathbb{Z}}$ of subgroups $M_n$ of the additive group $M^+$ such
that the following conditions hold:
\begin{enumerate}
\item[\rm (a)]$M=\bigoplus_{n\in \mathbb{Z}} M_n$.
\item[\rm (b)]For every $m\geq 0$ and $n\in \mathbb{Z}$, $B_mM_n\subseteq \bigoplus_{k\leq m+n}M_k$.
\end{enumerate}
The collection $\{M_n\}_{n\in \mathbb{Z}}$ is called a
semi-graduation of $M$ and we say that the elements of $M_n$ are
homogeneous of degree $n$. We say that $M$ is positively
semi-graded, also called $\mathbb{N}$-semi-graded, if $M_n=0$ for
every $n<0$. Let $f: M\to N$ be an homomorphism of $B$-modules,
where $M$ and $N$ are semi-graded $B$-modules; we say that $f$ is
homogeneous if $f(M_n)\subseteq N_n$ for every $n\in \mathbb{Z}$.

An important class of semi-graded rings that includes finitely
graded algebras is the following (see \cite{lezamalatorre}). Let
$B$ be a ring. We say that $B$ is \textit{finitely semi-graded}
$(FSG)$ if $B$ satisfies the following conditions:
\begin{enumerate}
\item[\rm (1)]$B$ is $SG$.
\item[\rm (2)]There exists finitely many elements $x_1,\dots,x_n\in B$ such that the
subring generated by $B_0$ and $x_1,\dots,x_n$ coincides with $B$.
\item[\rm (3)]For every $n\geq 0$, $B_n$ is a free $B_0$-module of finite dimension.
\end{enumerate}
Moreover, if $M$ is a $B$-module, we say that $M$ is finitely
semi-graded if $M$ is semi-graded, finitely generated, and for
every $n\in \mathbb{Z}$, $M_n$ is a free $B_0$-module of finite
dimension.

\begin{remark}
(i) It is clear that any $\mathbb{N}$-graded ring is $SG$.

(ii) Any finitely graded algebra is a $FSG$ ring.
\end{remark}

From the definitions above we get the following elementary facts.

\begin{proposition}[\cite{lezamalatorre}]\label{proposition17.5.5}
Let $B=\bigoplus_{n\geq 0}B_n$ be a $SG$ ring. Then,
\begin{enumerate}
\item[\rm (i)]$B_0$ is a subring of $B$. Moreover, for any $n\geq 0$, $B_0\oplus \cdots \oplus B_{n}$ is a $B_0-B_0$-bimodule, as well as $B$.
\item[\rm (ii)]$B$ has a standard $\mathbb{N}$-filtration given by
\begin{equation}\label{equ17.5.1}
F_n(B):=B_0\oplus \cdots \oplus B_{n}.
\end{equation}
\item[\rm (iii)]The associated graded ring $Gr(B)$ satisfies
\begin{center}
$Gr(B)_n\cong B_n$, for every $n\geq 0$ $($isomorphism of abelian
groups$)$.
\end{center}
\item[\rm (iv)]Let $M=\bigoplus_{n\in \mathbb{Z}}M_n$ be a semi-graded $B$-module and $N$ a submodule of $M$. The following conditions are equivalent:
\begin{enumerate}
\item[\rm (a)]$N$ is semi-graded.
\item[\rm (b)]For every $z\in N$, the homogeneous components of $z$ are in $N$.
\item[\rm (c)]$M/N$ is semi-graded with semi-graduation given by
\begin{center}
$(M/N)_n:=(M_n+N)/N$, $n\in \mathbb{Z}$.
\end{center}
\end{enumerate}
\end{enumerate}
\end{proposition}

\begin{remark}\label{remark18.1.7}
(i) If $B$ is a $FSG$ ring, then for every $n\geq 0$,
$Gr(B)_n\cong B_n$ as $B_0$-modules.

(ii) Observe if $B$ is $FSG$ ring, then $B_0B_p=B_p$ for every
$p\geq 0$, and if $M$ is finitely semi-graded, then $B_0M_n=M_n$
for all $n\in \mathbb{Z}$.
\end{remark}

\subsection{Skew $PBW$ extensions}

In order to present enough examples of $FSG$ rings not being
necessarily finitely graded algebras, we recall in this subsection
the notion of skew $PBW$ extension defined firstly in
\cite{LezamaGallego}.

\begin{definition}[\cite{LezamaGallego}]\label{gpbwextension}
Let $R$ and $A$ be rings. We say that $A$ is a \textit{skew $PBW$
extension of $R$} $($also called a $\sigma-PBW$ extension of
$R$$)$, if the following conditions hold:
\begin{enumerate}
\item[\rm (i)]$R\subseteq A$.
\item[\rm (ii)]There exist finitely many elements $x_1,\dots ,x_n\in A$ such $A$ is a left $R$-free module with basis
\begin{center}
${\rm Mon}(A):= \{x^{\alpha}=x_1^{\alpha_1}\cdots
x_n^{\alpha_n}\mid \alpha=(\alpha_1,\dots ,\alpha_n)\in
\mathbb{N}^n\}$, with $\mathbb{N}:=\{0,1,2,\dots\}$.
\end{center}
The set $Mon(A)$ is called the set of standard monomials of $A$.
\item[\rm (iii)]For every $1\leq i\leq n$ and $r\in R-\{0\}$, there exists $c_{i,r}\in R-\{0\}$ such that
\begin{equation}\label{sigmadefinicion1}
x_ir-c_{i,r}x_i\in R.
\end{equation}
\item[\rm (iv)]For every $1\leq i,j\leq n$, there exists $c_{i,j}\in R-\{0\}$ such that
\begin{equation}\label{sigmadefinicion2}
x_jx_i-c_{i,j}x_ix_j\in R+Rx_1+\cdots +Rx_n.
\end{equation}
Under these conditions we will write $A:=\sigma(R)\langle
x_1,\dots ,x_n\rangle$.
\end{enumerate}
\end{definition}
Associated to a skew $PBW$ extension $A=\sigma(R)\langle x_1,\dots
,x_n\rangle$, there are $n$ injective endomorphisms
$\sigma_1,\dots,\sigma_n$ of $R$ and $\sigma_i$-derivations
$\delta_i$, as the following proposition shows.
\begin{proposition}[\cite{LezamaGallego}]\label{sigmadefinition}
Let $A$ be a skew $PBW$ extension of $R$. Then, for every $1\leq
i\leq n$, there exists an injective ring endomorphism
$\sigma_i:R\rightarrow R$ and a $\sigma_i$-derivation
$\delta_i:R\rightarrow R$ such that
\begin{center}
$x_ir=\sigma_i(r)x_i+\delta_i(r)$,
\end{center}
for each $r\in R$.
\end{proposition}

Some particular cases of skew $PBW$ extensions are the following.

\begin{definition}[\cite{LezamaGallego}]\label{sigmapbwderivationtype}
Let $A$ be a skew $PBW$ extension.
\begin{enumerate}
\item[\rm (a)]
$A$ is quasi-commutative if the conditions {\rm(}iii{\rm)} and
{\rm(}iv{\rm)} in Definition \ref{gpbwextension} are replaced by
\begin{enumerate}
\item[\rm (iii')]For every $1\leq i\leq n$ and $r\in R-\{0\}$ there exists $c_{i,r}\in R-\{0\}$ such that
\begin{equation}
x_ir=c_{i,r}x_i.
\end{equation}
\item[\rm (iv')]For every $1\leq i,j\leq n$ there exists $c_{i,j}\in R-\{0\}$ such that
\begin{equation}
x_jx_i=c_{i,j}x_ix_j.
\end{equation}
\end{enumerate}
\item[\rm (b)]$A$ is bijective if $\sigma_i$ is bijective for
every $1\leq i\leq n$ and $c_{i,j}$ is invertible for any $1\leq
i<j\leq n$.
\end{enumerate}
\end{definition}

\begin{definition}\label{1.1.6}
Let $A$ be a skew $PBW$ extension of $R$ with endomorphisms
$\sigma_i$, $1\leq i\leq n$, as in Proposition
\ref{sigmadefinition}.
\begin{enumerate}
\item[\rm (i)]For $\alpha=(\alpha_1,\dots,\alpha_n)\in \mathbb{N}^n$,
$\sigma^{\alpha}:=\sigma_1^{\alpha_1}\cdots \sigma_n^{\alpha_n}$,
$|\alpha|:=\alpha_1+\cdots+\alpha_n$. If
$\beta=(\beta_1,\dots,\beta_n)\in \mathbb{N}^n$, then
$\alpha+\beta:=(\alpha_1+\beta_1,\dots,\alpha_n+\beta_n)$.
\item[\rm (ii)]For $X=x^{\alpha}\in Mon(A)$,
$\exp(X):=\alpha$ and $\deg(X):=|\alpha|$.
\item[\rm (iii)]Let $0\neq f\in A$, $t(f)$ is the finite
set of terms that conform $f$, i.e., if $f=c_1X_1+\cdots +c_tX_t$,
with $X_i\in Mon(A)$ and $c_i\in R-\{0\}$, then
$t(f):=\{c_1X_1,\dots,c_tX_t\}$.
\item[\rm (iv)]Let $f$ be as in {\rm(iii)}, then $\deg(f):=\max\{\deg(X_i)\}_{i=1}^t.$
\end{enumerate}
\end{definition}

The next theorems establish some results for skew $ PBW $
extensions that we will use later, for their proofs see
\cite{Oswaldo}, \cite{lezamalatorre} and \cite{LezamaHelbert}.

\begin{theorem}[\cite{Oswaldo}]\label{filteredskew}
Let $ A $ be an arbitrary skew $PBW$ extension of the ring $ R $.
Then, $ A $ is a filtered ring with filtration given by
$$F_{m}:=
\begin{cases}
R, & \text{ if } m=0\\
\{f \in A| deg(f) \leq m\}, & \text{  if } m \geq 1
\end{cases}$$
and the corresponding graded ring $ Gr(A) $ is a quasi-commutative
skew $ PBW $ extension of $ R $. Moreover, if $ A $ is bijective,
then $ Gr(A) $ is quasi-commutative bijective skew $ PBW $
extension of $ R $.
\end{theorem}
\begin{theorem}[\cite{Oswaldo}]\label{skewore}
Let $ A $ be a quasi-commutative skew $ PBW $ extension of a ring
$ R $. Then,
\begin{enumerate}
\item $ A $ is isomorphic to an iterated skew polynomial ring of endomorphism type, i.e.,
$$ A \cong R[z_{1};\theta_{1}]\dots[z_{n};\theta_{n}].$$
\item If $ A $ is bijective, then each endomorphism $ \theta_{i} $ is bijective, $ 1\leq i\leq n $.
\end{enumerate}
\end{theorem}
\begin{theorem}[Hilbert Basis Theorem, \cite{Oswaldo}] Let $ A $ be a bijective skew $ PBW $ extension of $ R $. If $ R $ is a left (right)
Noetherian ring then $ A $ is also a left (right) Noetherian ring.
\end{theorem}

\begin{theorem}[\cite{LezamaHelbert}]\label{skewstrongly}
Let $K$ be a field and let $A = \sigma(R)\langle
x_{1},\dots,x_{n}\rangle$ be a bijective skew $PBW$ extension of a
left strongly Noetherian $K$-algebra $R$. Then $A$ is left
strongly Noetherian.
\end{theorem}

\begin{theorem}[\cite{lezamalatorre}]\label{proposition16.5.7}
Any skew $PBW$ extension $A = \sigma(R)\langle
x_{1},\dots,x_{n}\rangle$ is a $FSG$ ring with semi-graduation
$A=\bigoplus_{k\geq 0} A_k$, where
\begin{center}
$A_k:=_R\langle x^\alpha\in Mon(A)|\deg(x^\alpha)=k\rangle$.
\end{center}
\end{theorem}

\begin{example}\label{example1.10}
Many important algebras and rings coming from mathematical physics
are particular examples of skew $PBW$ extensions, see
\cite{Oswaldo} and \cite{Reyes}: Habitual ring of polynomials in
several variables, Weyl algebras, enveloping algebras of finite
dimensional Lie algebras, algebra of $q$-differential operators,
many important types of Ore algebras, algebras of diffusion type,
additive and multiplicative analogues of the Weyl algebra, dispin
algebra $\mathcal{U}(osp(1,2))$, quantum algebra
$\mathcal{U}'(so(3,K))$, Woronowicz algebra
$\mathcal{W}_{\nu}(\mathfrak{sl}(2,K))$, Manin algebra
$\mathcal{O}_q(M_2(K))$, coordinate algebra of the quantum group
$SL_q(2)$, $q$-Heisenberg algebra \textbf{H}$_n(q)$, Hayashi
algebra $W_q(J)$, differential operators on a quantum space
$D_{\textbf{q}}(S_{\textbf{q}})$, Witten's deformation of
$\mathcal{U}(\mathfrak{sl}(2,K))$, multiparameter Weyl algebra
$A_n^{Q,\Gamma}(K)$, quantum symplectic space
$\mathcal{O}_q(\mathfrak{sp}(K^{2n}))$, some quadratic algebras in
3 variables, some 3-dimensional skew polynomial algebras, among
many others. For a precise definition of any of these rings and
algebras see \cite{Oswaldo} and \cite{Reyes}.
\end{example}

\section{Parametrization of point modules of $FSG$ rings}

In this section we complete the parametrization of point modules
of $FSG$ rings that was initiated in \cite{Lezama-Gomez}. The main
result of the present section is Theorem \ref{theorem17.6.8}
below. For completeness, we include the basic facts of
parametrization studied in \cite{Lezama-Gomez}, omitting the
proofs.

\begin{definition}\label{definition17.5.14}
Let $B=\bigoplus_{n\geq 0}B_n$ be a $FSG$ ring that is generated
in degree $1$.
\begin{enumerate}
\item[\rm (i)]A point module for $B$ is a finitely $\mathbb{N}$-semi-graded $B$-module $M=\bigoplus_{n\in \mathbb{N}}M_n$
such that $M$ is cyclic, generated in degree $0$, i.e., there
exists an element $m_0\in M_0$ such that $M=Bm_0$, and
$\dim_{B_0}(M_n)=1$ for all $n\geq 0$.
\item[\rm (ii)]Two point modules $M$ and $M'$ for $B$ are isomorphic if there exists a homogeneous $B$-isomorphism between them.
\item[\rm (iii)]$P(B)$ is the collection of isomorphism classes of point modules for $B$.
\end{enumerate}
\end{definition}

The following result is the first step in the construction of the
geometric structure for $P(B)$.

\begin{theorem}\label{theorem16.5.2}
Let $B=\bigoplus_{n\geq 0}B_n$ be a $FSG$ ring generated in degree
$1$. Then, $P(B)$ has a Zariski topology generated by finite
unions of sets $V(J)$ defined by {\small
\begin{center}
$V(J):=\{M\in P(B)\mid Ann(M)\supseteq J\}$,
\end{center}}
where $J$ ranges the semi-graded left ideals of $B$.
\end{theorem}
\begin{proof}
See \cite{Lezama-Gomez}, Theorem 9.
\end{proof}

\begin{definition}
Let $B=\bigoplus_{n\geq 0}B_n$ be a $FSG$ ring generated in degree
$1$ such that $B_0$ is commutative and $B$ is a $B_0$-algebra. Let
$S$ be a commutative $B_0$-algebra. A $S$-point module for $B$ is
a $\mathbb{N}$-semi-graded $S\otimes_{B_0} B$-module $M$  which is
cyclic, generated in degree $0$, $M_n$ is a locally free
$S$-module with ${\rm rank}_S(M_n)=1$ for all $n\geq 0$, and
$M_0=S$. $P(B;S)$ will denote the set of $S$-point modules for
$B$.
\end{definition}

\begin{theorem}\label{theorem16.5.5}
Let $B=\bigoplus_{n\geq 0}B_n$ be a $FSG$ ring generated in degree
$1$ such that $B_0$ is commutative and $B$ is a $B_0$-algebra. Let
$\mathcal{B}_0$ be the category of commutative $B_0$-algebras and
let $\mathcal{S}et$ be the category of sets. Then, $P$ defined by
\begin{align*}
\mathcal{B}_0 & \xrightarrow{P} \mathcal{S}et\\
S & \mapsto P(B;S)\\
S\to T &\mapsto P(B;S)\to P(B;T), \ \text{given by}\ M\mapsto
T\otimes_S M
\end{align*}
is a covariant functor called the point functor for $B$.
\end{theorem}
\begin{proof}
See \cite{Lezama-Gomez}, Theorem 10.
\end{proof}

\begin{definition}
Let $B=\bigoplus_{n\geq 0}B_n$ be a $FSG$ ring generated in degree
$1$ such that $B_0$ is commutative and $B$ is a $B_0$-algebra. We
say that a $B_0$-scheme $X$ parametrizes the point modules of $B$
if the point functor $P$ is naturally isomorphic to $h_X$.
\end{definition}

Taking $S=B_0$ we get that $P(B)\subseteq P(B;B_0)$. If $B_0=K$ is
a field, then clearly $P(B)=P(B;K)$.

\begin{theorem}\label{proposition17.6.7}
Let $B=\bigoplus_{n\geq 0}B_n$ be a $FSG$ ring generated in degree
$1$ such that $B_0=K$ is a field and $B$ is a $K$-algebra. Let $X$
be a $K$-scheme that parametrizes $P(B)$. Then, there exists a
bijective correspondence between the closed points of $X$ and
$P(B)$.
\end{theorem}
\begin{proof}
See \cite{Lezama-Gomez}, Theorem 11.
\end{proof}

\begin{theorem}\label{theorem17.6.8}
Let $B=\bigoplus_{n\geq 0}B_n$ be a $FSG$ ring generated in degree
$1$ such that $B_0=K$ is a field and $B$ is a $K$-algebra. Then,
\begin{enumerate}
\item[\rm (i)]There is an injective function
\begin{center}
$(P(B)/\sim) \xrightarrow{\alpha} P(Gr(B))$,
\end{center}
where $\sim $ is the relation in $P(B)$ defined by
\begin{center}
$[M]\sim [M']\Leftrightarrow Gr(M)\cong Gr(M')$,
\end{center}
with $[M]$ the class of point modules isomorphic to the point
module $M$.
\item[\rm (ii)] Let $X$ be a $K$-scheme that parametrizes
$P(Gr(B))$. Then there is an injective function from $(P(B)/\sim)$
to the closed points of $X$:
\begin{center}
$(P(B)/\sim) \xrightarrow{\widetilde{\alpha}} X$.
\end{center}
\end{enumerate}
\end{theorem}
\begin{proof}
(i) We divide the proof of this part in three steps.

\textit{Step 1}. Note that $Gr(B)$ is a finitely graded
$K$-algebra generated in degree $1$: According to Proposition
\ref{proposition17.5.5}, $Gr(B)$ is $\mathbb{N}$-graded with
$Gr(B)_n\cong B_n$ for all $n\geq 0$ (isomorphism of $K$-vector
spaces), in particular, $Gr(B)_0=K$; if $x_1,\dots,x_m\in B_1$
generate $B$ as $K$-algebra, then
$\overline{x_1},\dots,\overline{x_m}\in Gr(B)_1$ generate $Gr(B)$
as $K$-algebra.

\textit{Step 2}. Let $[M]\in P(B)$, with $M=\bigoplus_{n\geq
0}M_n$, $M=Bm_0$, with $m_0\in M_0$, and $\dim_{K}(M_n)=1$ for all
$n\geq 0$. As in Proposition \ref{proposition17.5.5}, we define
\begin{center}
$Gr(M):=\bigoplus_{n\geq 0}Gr(M)_n$,

$Gr(M)_n:=\frac{M_0\oplus \cdots \oplus M_n}{M_0\oplus \cdots
\oplus M_{n-1}}\cong M_n$ (isomorphism of $K$-vector spaces).
\end{center}
The structure of $Gr(B)$-module for $Gr(M)$ is given by
bilinearity and the product
\begin{center}
$\overline{b_p}\cdot \overline{m_n}:=\overline{b_pm_n}\in
Gr(M)_{p+n}$.
\end{center}
This product is well-defined: If $\overline{b_p}=\overline{c_p}$,
then $b_p-c_p\in B_0\oplus \cdots \oplus B_{p-1}$ and hence
$(b_p-c_p)m_n\in M_0\oplus \cdots \oplus M_{p-1+n}$, so in
$Gr(M)_{p+n}$ we have $\overline{(b_p-c_p)m_n}=\overline{0}$,
i.e., $\overline{b_pm_n}=\overline{c_pm_n}$; now, if
$\overline{m_n}=\overline{l_n}$, then $m_n-l_n\in M_0\oplus \cdots
\oplus M_{n-1}$, whence $b_p(m_n-l_n)\in M_0\oplus \cdots \oplus
M_{p+n-1}$, so in $Gr(M)_{p+n}$ we have
$\overline{b_p(m_n-l_n)}=\overline{0}$, i.e.,
$\overline{b_pm_n}=\overline{b_pl_n}$. The associative law for
this product holds:
\begin{center}
$\overline{b_q}\cdot (\overline{b_p}\cdot
\overline{m_n})=\overline{b_q}\cdot
(\overline{b_pm_n})=\overline{b_q(b_pm_n)}=\overline{(b_qb_p)m_n}=\overline{(c_0+\cdots+c_{p+q})m_n}
=\overline{c_{p+q}m_n}=(\overline{b_q}\
\overline{b_p})\overline{\cdot m_n}$.
\end{center}
Finally, $\overline{1}\cdot
\overline{m_n}=\overline{m_n}$.

Observe that $Gr(M)=Gr(B)\cdot \overline{m_0}$: Let
$\overline{m_n}\in Gr(M)_n$, with $m_n\in M_n$, we have
$\overline{m_n}=\overline{bm_0}=\overline{(b_0+\cdots+b_n)m_0}=\overline{b_0m_0+\cdots+b_{n-1}m_0+b_nm_0}=\overline{b_nm_0}
=\overline{b_n}\cdot\overline{m_0}$.

\textit{Step 3}. We define
\begin{align*}
P(B) & \xrightarrow{\alpha'} P(Gr(B))\\
[M]\ & \mapsto [Gr(M)]
\end{align*}
It is clear that $\alpha'$ is well-defined, i.e., if $M\cong M'$,
then $Gr(M)\cong Gr(M')$.

Note that the relation $\sim$ defined in the statement of the
theorem is an equivalence relation, let $[[M]]$ be the class of
$[M]\in P(B)$, then we define
\begin{align*}
(P(B)/\sim) & \xrightarrow{\alpha} P(Gr(B))\\
[[M]]\ & \mapsto \alpha'([M])=[Gr(M)].
\end{align*}
It is clear that $\alpha$ is a well-defined injective function.

(ii) This follows from (i) and applying Theorem
\ref{proposition17.6.7} to $Gr(B)$.
\end{proof}



\section{Point modules of skew $PBW$ extensions}

In this final section we apply the results of the previous section
to compute the point modules of some skew $PBW$ extensions, and
hence, of some important examples of non $\mathbb{N}$-graded
quantum algebras.

\begin{theorem}\label{corollary17.6.9}
Let $K$ be a field and $A=\sigma(K)\langle x_1,\dots,x_n\rangle$
be a bijective skew $PBW$ extension such that $A$ is a
$K$-algebra. Then, there exists $m_0\geq 1$ such that there is an
injective function from $P(A)/\sim$ to the closed points of
$X_{m_0}$:
\begin{center}
$(P(A)/\sim) \longrightarrow X_{m_0}$.
\end{center}
\end{theorem}
\begin{proof}
From Proposition \ref{proposition16.5.7}, $A$ is a $FSG$ ring
generated in degree $1$ with $A_0=K$, and by hypothesis, $A$ is a
$K$-algebra. Observe that the filtration in Theorem
\ref{filteredskew} coincides with the one in (\ref{equ17.5.1}),
thus, from Theorem \ref{skewore} we get that $Gr(A)$ is the
$n$-multiparametric quantum space $K_{\textbf{q}}[x_1,\dots,x_n]$
(in $A$, $x_ir=rx_i$, for all $r\in K$ and $1\leq i\leq n$). It is
clear that $Gr(A)$ is a finitely graded algebra with finite
presentation as in Proposition \ref{17.3.6}, moreover, $Gr(A)$ is
strongly noetherian (Theorem \ref{skewstrongly}), so the claimed
follows from Corollary \ref{corollaryA.3.8} and Theorem
\ref{theorem17.6.8}.
\end{proof}

For the exact definition of any of the quantum algebras in the
following examples see \cite{Oswaldo} and \cite{Reyes}.

\begin{example}\label{example3.2}
The following examples of skew $PBW$ extensions satisfy the
hypothesis of Theorem \ref{corollary17.6.9}, and hence, for them
there is $m_0\geq 1$ such that there exists an injective function
from $P(A)/\sim $ to the closed points of $X_{m_0}$.

\begin{enumerate}

\item The enveloping algebra of a Lie algebra $\mathcal{G}$ of
dimension $n$: In this case $Gr(A)=K[x_0,\dots,x_{n-1}]$, so
$m_0=1$ and $X_1=\mathbb{P}^{n-1}$ (Example \ref{17.3.2}). We have
a similar situation for the algebra $S_h$ of shift operators: In
this case $Gr(A)=K[x_0,x_{1}]$, so $m_0=1$ and
$X_1=\mathbb{P}^{1}$.

\item Quantum algebra $\mathcal{U}'(so(3,K))$, with $q\in
K-\{0\}$: Since in this case $Gr(A)$ is a $3$-multiparametric
quantum affine space, then from Example \ref{example17.6.11} we
get that $m_0=2$ and $X_2$ is as in (\ref{equation17.6.10}). We
have similar descriptions for the dispin algebra $\cU(osp(1,2))$,
the Woronowicz algebra $\cW_{\nu}(\mathfrak{sl}(2,K))$, where $\nu
\in K-\{0\}$ is not a root of unity, nine types of $3$-dimensional
skew polynomial algebras, and the multiplicative analogue of the
Weyl algebra with $3$ variables.
\item For the following algebras,
$Gr(A)=K_{\textbf{q}}[x_1,\dots,x_n]$, so from Example
\ref{remark17.6.11} we conclude that $m_0=2$ and $X_2$ is as in
(\ref{equation17.6.4}): The $q$-Heisenberg algebra, the algebra of
linear partial $q$-dilation operators, the algebra $D$ for
multidimensional discrete linear systems, the algebra of linear
partial shift operators.
\end{enumerate}
\end{example}

\begin{example}\label{example3.3}
The following algebras are not skew $PBW$ extensions of the base
field $K$, however, they satisfy the hypothesis of Corollary
\ref{corollaryA.3.8}, and hence, for them there is $m_0\geq 1$
such that $X_{m_0}$ parametrizes $P(A)$ and there exists a
bijective correspondence between the closed points of $X_{m_0}$
and $P(A)$: The algebras of diffusion type, the algebra
$\textbf{U}$, the Manin algebra, or more generally, the algebra
$\mathcal{O}_q(M_n(K))$ of quantum matrices, some quadratic
algebras in $3$ variables, the quantum symplectic space
$\mathcal{O}_q(\mathfrak{sp}(K^{2n}))$.
\end{example}


\end{document}